\documentclass{amsart}

\usepackage{amsmath}
\usepackage{amsthm,amsfonts,amssymb,mathrsfs}
\usepackage{epic,eepic}
\usepackage{yfonts}
\usepackage{paralist,enumerate}
\usepackage[all]{xy}
\usepackage{graphicx}

\newtheorem{theorem}{Theorem}[subsection]
\newtheorem{proposition}{Propostion}[subsection]
\newtheorem{lemma}{Lemma}[subsection]

\theoremstyle{definition}
\newtheorem{definition}{Definition}[subsection]

\theoremstyle{remark}
\newtheorem{remark}{Remark}[subsection]

\newcommand{\Spec}{{\mathrm{Spec}}}

\begin{document}

\title[Two desingularization of the Kontsevich's moduli space]{Comparison of two desingularization of the Kontsevich's moduli space of elliptic stable maps}
\author{Hyenho Lho}
\address{Department of Mathematics, Seoul National University, Seoul 151-747, Korea}
\email{nhh128@snu.ac.kr}

\begin{abstract}
It is known that the main component of the Kontsevich's moduli space of elliptic stable maps is singular. There are two different desingularizations. One is Vakil-Zinger's desingularization and the other is the moduli space of logarithmic stable maps. When the degree is less then or equal to $3$ and the target is $\mathbb{P}^{n}$, we show that the moduli space of logarithmic stable maps can be obtained by blowing up Vakil-Zinger's desingularization.
\end{abstract}

\maketitle

%%%%%%%%%%%%%%%%%%%%%%%%%%%%%%%%%%%%%%%
\section{Introduction}
\bigskip

The Kontsevich's moduli space of stable maps  $\mathbf{\overline{M}}_{g,k}(X,d)$ is a moduli space which parametrizes maps from k-marked nodal curve of arithmetic genus $g$ to projective variety $X$ satisfying stability conditions. See \cite{FP} for precise definitions and properties. In this paper we only consider Kontsevich's moduli space of elliptic stable maps $\mathbf{\overline{M}}_{1,0}(\mathbb{P}^{n},d)$. $\mathbf{\overline{M}}_{1,0}(\mathbb{P}^{n},d)$ is known to have several components. We call the component parametrizing elliptic stable maps whose domain curve have non-contracted elliptic subcurves the main component. We denote the main component of $\mathbf{\overline{M}}_{1,0}(\mathbb{P}^{n},d)$ as $\mathbf{\overline{M}}_{1,0}(\mathbb{P}^{n},d)_{0}$. It is known that $\mathbf{\overline{M}}_{1,0}(\mathbb{P}^{n},d)_{0}$ is singular.

 Recently many birational model of $\mathbf{\overline{M}}_{1,0}(\mathbb{P}^{n},d)_{0}$ have been introduced by many authors.
 In \cite{VZ}, Vakil and Zinger found a canonical desingularization $\mathbf{\widetilde{M}}_{1,0}(\mathbb{P}^{n},d)_{0}$ of $\mathbf{\overline{M}}_{1,0}(\mathbb{P}^{n},d)_{0}$ by blowing-up $\mathbf{\overline{M}}_{1,0}(\mathbb{P}^{n},d)_{0}$. 
In \cite{Kim}, Kim introduced another desingularization of $\mathbf{\overline{M}}_{1,0}(\mathbb{P}^{n},d)_{0}$ called the moduli space of logarithmic stable maps by using log structures. We denote this space as $\overline{\mathbf{M}}_{1,0}^{log,ch}(\mathbb{P}^{n},d)$. In \cite{MOP}, Marian, Oprea and Pandharipande constructed moduli space of stable quotients denoted by $\mathbf{Q}_{g}(\mathbb{P}^{n},d)$. They defined a moduli space of stable quotients of the rank $n$ trivial sheaf on nodal curves. They also proved that when the genus is 1, $\mathbf{Q}_{1}(\mathbb{P}^{n},d)$ is a smooth Delign-Mumford stack. So this gives another smooth birational model.
In \cite{Vis}, Viscardi constructed a moduli space of ($m$)-stable maps denoted by $\overline{\mathbf{M}}_{1,k}^{(m)}(\mathbb{P}^{n},d)$. He defined a moduli space using ($m$)-stable curves which was introduced by Smyth \cite{S}. He also proved  $\overline{\mathbf{M}}_{1,k}^{(m)}(\mathbb{P}^{n},d)$ is smooth if $d+k \leq m \leq 5$.  

In general, it is not known how these birational moduli spaces are related to each others.
In this paper, we compare Vakil-Zinger's desingularization and the moduli space of logarithmic stable mpas. We show that $\overline{\mathbf{M}}_{1,0}^{log,ch}(\mathbb{P}^{n},3)$ can be obtained by blowing up $\mathbf{\widetilde{M}}_{1,0}(\mathbf{P}^{n},3)_{0}$ along the locus $\sum_{2}$, $\Gamma_{2}$, $\sum_{1}$, $\Gamma_{1}$.

${\sum_{1}}$ is the closure of the locus of $\mathbf{\overline{M}}_{1,0}(\mathbb{P}^{n},d)_{0}$ parametrizing stable maps such that their domain curves consist of a elliptic component of the degree $0$ and a rational component of the degree $3$ and the morphism restricted to the rational component has ramification order $3$ at the nodal point.

${\sum_{2}}$ is the closure of the locus of $\mathbf{\overline{M}}_{1,0}(\mathbb{P}^{n},d)_{0}$ parametrizing stable maps such that their domain curves consist of a elliptic component of the degree $0$, and two rational components with the degree $1$,$2$, each meeting the elliptic component at one point and the morphism restricted to degree $2$ rational component has ramification order $2$ at the nodal point.

$\Gamma_{1}$ is the closure of the locus of $\mathbf{\overline{M}}_{1,0}(\mathbb{P}^{n},d)_{0}$ parametrizing stable maps such that their domain curves consist of a elliptic component of the degree $0$ and a rational component of the degree three, and there exists a smooth point $q$ on the rational component such that $p$, $q$ go to same point, where $p$ is the node point.

$\Gamma_{2}$ is the closure of the locus of $\mathbf{\overline{M}}_{1,0}(\mathbb{P}^{n},d)_{0}$ parametrizing stable maps such that their domain curves consist of a elliptic component of the degree $0$ and two rational components with the degree $1$, $2$, each meeting the elliptic component at one point and there exists smooth point $q$ on degree $2$ rational component such that $p$, $q$ go to same point where $p$ is nodal point on degree $2$ rational component.

The outline of this paper is as follows.
In section 2, we give some preliminaries.
In section 3, we present an example of a degeneration where a nontrivial elliptic logarithmic stable map occurs.
In section 4, we calculate the fiber of the natural morphism from the moduli space of admissible stable maps to the Kontsevich's moduli space of stable maps.
In section 5, we prove two moduli spaces are equal if the degree is $2$. 
In section 6, we describe etale charts of $\mathbf{\overline{M}}_{1,0}(\mathbb{P}^{n},3)_{0}$ explicitly and by blowing up suitable subschemes, we obtain
\begin{theorem}
$\overline{\mathbf{M}}_{1,0}^{log,ch}(\mathbb{P}^{n},3)$ can be obtained by blowing-up $\mathbf{\widetilde{M}}_{1,0}(\mathbb{P}^{n},3)_{0}$ along the locus $\sum_{2}$, $\Gamma_{2}$, $\sum_{1}$, $\Gamma_{1}$.
\end{theorem}\

\noindent
{\bf Acknowledgement} I would like to thank my advisor, Young-Hoon Kiem, for giving me this problem and all of the guidance and support during this work.\

\bigskip
\bigskip
\bigskip
%%%%%%%%%%%%%%%%%%%%%%%%%%%%%%%%%%%%%%%
\section{Preliminaries}
\bigskip

In this section we introduce some notations. We also briefly recall some definitions and properties of Vakil-Zinger's desingularization and the moduli space of logarithmic stable maps.

\subsection{Notations}
\subsubsection{A dual graph of domain curves} In this paper we only consider connected curves of arithmetic genus $1$. Note that every connected curve of arithmetic genus 1 has the unique minimal subcurve of arithmetic genus $1$. we give names to this subcurve.
\begin{definition} Let $C$ be connected curve of arithmetic genus 1. Let $C'$ be the minimal subcurve of arithmetic genus 1 of $C$. We call $C'$ the {\em essential part} of $C$.\
\end{definition}

For every nodal curve, we can associate a graph called the dual graph. Irreducible components of the nodal curve correspond to vertices of the graph. And nodal points of nodal curve correspond to edges of graph.

If curve $C$ is connected curve of arithmetic genus $1$ whose essential part is irreducible curve, we can represent its dual graph as following. Suppose $C$ has $6$ irreducible components $E$, $C_{1}$, $C_{2}$, $B_{1}$, $B_{2}$, $B_{3}$. $E$ is a smooth curve of arithmetic genus 1. Two smooth rational components $C_{1}$, $C_{2}$  are connected to $E$. And three smooth rational components $B_{1}$, $B_{2}$, $B_{3}$ are connected to $C_{1}$. Then we can represent the dual graph of $C$ as $E[C_{1}[B_{1},B_{2},B_{3}],C_{2}]$. In this case, we say $C$ is of the type $E[C_{1}[B_{1},B_{2},B_{3}],C_{2}]$. We denote the intersection point of $E$ and $C_{1}$ as $c_{1}$. And we denote intersection point of $C_{1}$ and $B_{1}$ as $b_{1}$ and so on.

\begin{figure}[here] 
\begin{center} 
\includegraphics[height=2in]{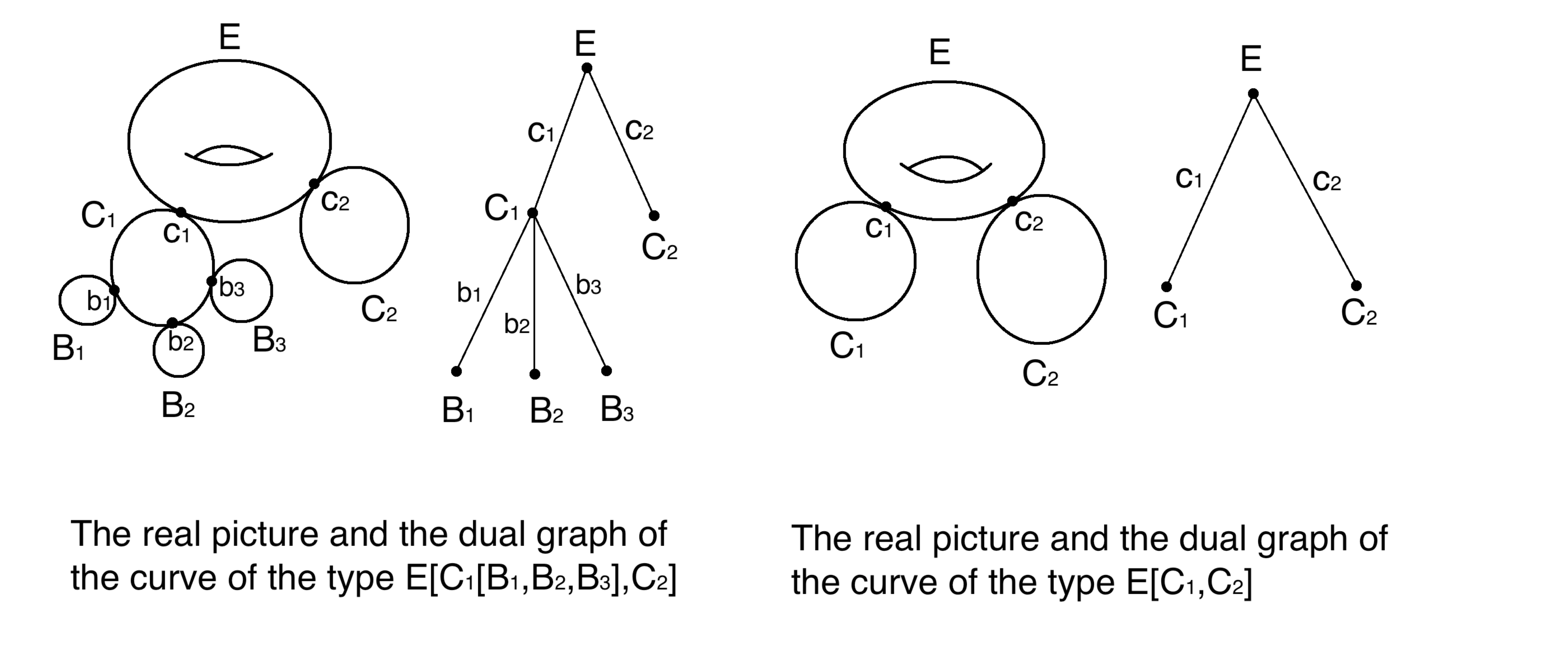} 

\end{center} 
\end{figure} 

Furthermore, if a curve $C$ is the domain curve of the Kontsevich's moduli space of elliptic stable maps, we record information of the degree in the parenthesis. For example, if we say that $C$ is of the type $E(0)[B_{1}(0)[C_{1}(1),C_{2}(2)]]$, then the dual graph of $C$ is represented as $E[B_{1}[C_{1},C_{2}]]$ and the degrees of maps restricted to $E$, $B_{1}$, $C_{1}$, $C_{2}$ are $0$, $0$, $1$, $2$, respectively. 

\subsubsection{The expanded target}
Let $\mathbb{P}^{n}$ be a n-dimensional projective space.
We define $\mathbb{P}^{n}(1)$ to be $(Bl_{c(0)}\mathbb{P}^{n})\bigcup\mathbb{P}^{n} $. Here $c(0)$ is a point in $\mathbb{P}^{n}$. And $Bl_{c(0)}\mathbb{P}^{n}$ and $\mathbb{P}^{n}$ are glued along $D(1)$. In $Bl_{c(0)}\mathbb{P}^{n}$, $D(1)$ is the exceptional divisor. And in $\mathbb{P}^{n}$, $D(1)$ is a hyperplane. We can give the linear order to the set of irreducible components of $\mathbb{P}^{n}(1)$ such that component corresponding to $\mathbb{P}^{n}$ is the largest one. We denote the irreducible components of $\mathbb{P}^{n}(1)$ by $\mathbb{P}^{n}_{1}$, $\mathbb{P}^{n}_{2}$ according to this order. i.e. $\mathbb{P}^{n}_{2}$ is the largest one.

We define $\mathbb{P}^{n}(2)$ to be $(Bl_{c(1)}\mathbb{P}^{n}(1))\bigcup\mathbb{P}^{n} $. Here $c(1)$ is a point in $\mathbb{P}^{n}_{2}$ not contained in D(1). And $Bl_{c(1)}\mathbb{P}^{n}(1)$ and $\mathbb{P}^{n}$ are glued along $D(2)$. In $Bl_{c(0)}\mathbb{P}^{n}(1)$, $D(2)$ is the exceptional divisor. And in $\mathbb{P}^{n}$, $D(2)$ is a hyperplane. We can give the linear order to the set of irreducible components of $\mathbb{P}^{n}(2)$ such that component corresponding to $\mathbb{P}^{n}$ is the largest one. We denote the irreducible components of $\mathbb{P}^{n}(2)$ by $\mathbb{P}^{n}_{1}$, $\mathbb{P}^{n}_{2}$, $\mathbb{P}^{n}_{3}$ according to this order. i.e. $\mathbb{P}^{n}_{3}$ is the largest one.

In this way, we define $\mathbb{P}^{n}(k)$, $\mathbb{P}^{n}_{1}$, $\mathbb{P}^{n}_{2}$, $\cdots$, $\mathbb{P}^{n}_{k+1}$, $D(1)$, $D(2)$, $\cdots$, $D(k)$ inductively.

\subsubsection{The sequence of blow up} 
Let $X$ be an algebraic scheme. Let $V_{1}$, $V_{2}$, $\cdots$, $V_{n}$ be subschemes of $X$. When we say that we blow up $X$ along $V_{1}$, $V_{2}$, $\cdots$, $V_{n}$, we mean that we first blow up $V_{1}$, blow up the proper transform of $V_{2}$, $\cdots$ , and blow up the proper transform of $V_{n}$.\

By abuse of notation, we identify an ideal $J$ with the subscheme $V_{J}$ defined by $J$.

\bigskip

\subsection{Vakil-Zinger Desingularization}

In \cite{VZ0}, Vakil and Zinger defined the $m$-tail locus of $\mathbf{\overline{M}}_{1,0}(\mathbb{P}^{n},d)_{0}$ to be the locus parametrizing maps such that in the domain the contracted elliptic curve meets the rest of the curve a total of precisely m points. Desingularization is described as following way; blow up the $1$-tail locus, then the proper transform of the $2$-tail locus, etc. This process stop at finite steps, and resulting space  $\mathbf{\widetilde{M}}_{1,0}(\mathbb{P}^{n},d)_{0}$ is smooth Delign-Mumford stack.\

In \cite{HL},  Hu and Li described local equations of $\mathbf{\overline{M}}_{1.0}(\mathbb{P}^{n},d)$. They first defined the terminally weighted tree $\gamma$. To each $\gamma$, they associated the variety $Z_{\gamma}$ called local model and the subvariety $Z^{0}_{\gamma} \subset Z_{\gamma}$ called the type $\gamma$ loci in $Z_{\gamma}$. They defined DM-stack $S$ to have singularity type $\gamma$ at a closed point $s\in S$ if there is a scheme $Y$, a point $y\in Y$ and two smooth morphisms $q_{1} : Y \rightarrow S$, $q_{2} : Y \rightarrow Z_{\gamma}$
such that $q_{1}(y) = s$ and $q_{2}(y) \in Z^{0}_{\gamma}$. To each element $[u] \subset \mathbf{\overline{M}}_{1,0}(\mathbb{P}^{n},d)$, they associated terminally weighted rooted tree. They defined the substack $\mathbf{\overline{M}}_{1,0}(\mathbb{P}^{n},d)_{\gamma} \subset \mathbf{\overline{M}}_{1,0}(\mathbb{P}^{n},d)$ to be the subset of all $[u] \subset \mathbf{\overline{M}}_{1,0}(\mathbb{P}^{n},d)$ whose associated terminally weighted rooted trees is $\gamma$. Finally they showed that the stack $\mathbf{\overline{M}}_{1,0}(\mathbb{P}^{n},d)$ has singularity type $\gamma$ along $\mathbf{\overline{M}}_{1,0}(\mathbb{P}^{n},d)_{\gamma}$

We do not present full details here since our case is quite simple. When $d=3$, $1$-tail locus $D_{1} \subset \mathbf{\overline{M}}_{1,0}(\mathbb{P}^{n},3)_{0}$ and $2$-tail locus $D_{2} \subset \mathbf{\overline{M}}_{1,0}(\mathbb{P}^{n},3)_{0}$ are smooth divisors. $3$-tail locus $D_{3} \subset \mathbf{\overline{M}}_{1,0}(\mathbb{P}^{n},3)_{0}$ has description as follows. Let $Z$ be $\{(a_{1},a_{2},\cdots,a_{n-1},b_{1},b_{2},\cdots,b_{n-1},z_{1},z_{2}) \in \mathbb{A}^{2n} : a_{1}z_{1}-b_{1}z_{2} = a_{2}z_{1}-b_{2}z_{2} = \cdots = a_{n-1}z_{1}-b_{n-1}z_{2}=0 \}$, where $\mathbb{A}^{n}$ is $n$-dimensional affine space. Let $Z^{0} \subset Z$ to be $\{(a_{1},a_{2},\cdots,a_{n-1},b_{1},b_{2},\cdots,b_{n-1},z_{1},z_{2}) \in Z : z_{1} = z_{2} = 0 \}$. To each element $[u]\subset D_{3}$, there is a scheme $Y$, a point $y\in Y$ and two smooth morphisms $q_{1} : Y \rightarrow \mathbf{\overline{M}}_{1.0}(\mathbb{P}^{n},d)_{0}$, $q_{2} : Y \rightarrow Z$ such that $q_{1}(y) = [u]$ and $q_{2}(y) \in Z^{0}$.
Therefore $\mathbf{\widetilde{M}}_{1,0}(\mathbb{P}^{n},3)_{0}$ = $Bl_{D_{3}}\mathbf{\overline{M}}_{1,0}(\mathbb{P}^{n},3)_{0}$.

\bigskip

\subsection{Logarithmic stable maps}

We briefly introduce logarithmic stable maps following \cite{Kim}. There is standard reference for definition and some properties of the log structures (\cite{Kato}). We do not give full details about the log structures since the log structures are not used extensively in this paper.

\begin{definition} An algebraic space $W$ over $S$ is called a {\em Fulton-Macpherson (FM) type space} if
\begin{enumerate}  
            \item $W \rightarrow S$ is a proper, flat morphism;
            \item for every closed point $s\in S$, etale locally there is an etale morphism\\
            $$W_{\overline{s}} \rightarrow Spec(k({\overline{s}})[x,y,z_{1},z_{2},\cdots,z_{k-1}]/(xy))$$
            where $x$, $y$ and $z_{i}$ are indeterminates.
\end{enumerate}
\end{definition}

\begin{definition}[\cite{Kim},5.1.1]
A triple
$((C/S, \mathbf{p}), W/S, f: C \longrightarrow W)$
 is called a $n$-pointed, genus $g$, {\em admissible map}
to a FM type space $W/S$ if

\begin{enumerate}

\item $(C/S, {\mathbf{p}}= (p_1,...,p_n))$ is a $n$-pointed, genus $g$, prestable curve over $S$.

\item $W/S$ is a FM type space.

\item $f:C\longrightarrow W$ is a map over $S$.

\item ({\em Admissibility}) If a point $p\in C$ is mapped into the relatively singular locus $(W/S)^{\mathrm{sing}}$ of $W/S$,
then \'etale locally at $\bar{p}$, $f$ is factorized as
\[ \xymatrix{ C  \ar[dd]_{f}\ar[dr]&         & U\ar[ll]\ar[rr] \ar[rd] \ar[dd] &    & \Spec(A[u,v]/(uv- t)) \ar[ld]  \ar[dd]\\
                                                         & S   &     & \Spec A \ar'[l][ll]  &                          \\
                       W   \ar[ur] &       &   V \ar[ll]\ar[rr] \ar[ur]&     & \Spec A[x,y, z_1,...,z_{r-1}]/(xy-\tau ) \ar[ul]
      } \]
     where all 5 horizontal maps
 are formally \'etale; $u, v, x, y, z_i$ are indeterminates;
  $x=u^l$, $y=v^l$ under the far right vertical map for some positive integer $l$;
 $t, \tau$ are elements in the maximal ideal $\mathfrak{m}_A$ of
 the local ring $A$; and $\bar{p}$ is mapped to the point defined by the ideal $(u,v,\mathfrak{m}_A)$.

\end{enumerate}
\end{definition}

A log morphism $(W, M_W) /(S,N)$ is called an extended log twisted FM type space if $W\rightarrow S$ is FM type space and $M_{W}$, $N$ are log structures on $W$, $S$ satisfying some conditions.

\begin{definition}[\cite{Kim},5.2.2]
A log morphism $\left( f: (C,M_C, {\mathbf{p}})\longrightarrow(W, M_W)\right) /(S,N) $
      is called a $(g,n)$ {\em logarithmic prestable map over $(S,N)$}
       if

\begin{enumerate}

    \item $((C,M)/(S,N), {\mathbf{p}})$ is a $n$-pointed, genus g, minimal log prestable curve.

    \item $(W, M_W) /(S,N)$ is an extended log twisted FM type space.

    \item  ({\em Corank = \# Nondistinguished Nodes Condition})
              For every  $s\in S$, the rank of $\mathrm{Coker}(N^{W/S}_{\bar{s}} \longrightarrow N_{\bar{s}})$ coincides with
                the number of nondistinguished nodes on $C_{\bar{s}}$.

    \item $f: (C,M_{C})\longrightarrow (W, M_{W})$ is a log morphism over $(S,N)$.

    \item\label{LogAdmissible} ({\em Log Admissibility}) either of the following conditions, equivalent
    under the above four conditions, holds:

    \begin{itemize}

    \item $\underline{f}$ is admissible.

    \item $f^b: f^*M_{W}\longrightarrow M_{C}$ is simple at every distinguished node.

    \end{itemize}

   \end{enumerate}
\end{definition}

\begin{definition}[\cite{Kim},8.1]
Let $\mathbf{\overline{M}}_{1,0}^{log,ch}(X,d)$ be
the moduli stack of $(g=1,n=0,d\ne 0)$ logarithmic stable maps $(f,C,W)$
satisfying the following conditions additional to those in
Definition 3.0.2. For every $s\in S$,
\begin{enumerate}
\item
Every end component of $W_{\bar{s}}$  contains  the entire
image of  the essential part of $C_{\bar{s}}$ under $f_{\bar{s}}$.

\item The image of the essential part of $C_{\bar{s}}$  is nonconstant.

\end{enumerate}
\end{definition}
Here, it is possible that some of irreducible components in the essential part are mapped to points.
Note that the dual graph of the target $W_s$ must be a chain.
Such a log stable map is called an elliptic log stable map to a chain type FM space $W$ of the
smooth projective variety $X$.

\begin{theorem}[\cite{Kim},Main Theorem B]
The moduli stack $\mathbf{\overline{M}}_{1,0}^{log,ch}(X,d)$ of elliptic logarithmic stable maps to chain type FM spaces of $X$ is a proper Delign-Mumford stack. When $X$ is a projective space $\mathbb{P}^{n}$, the stack is smooth.
\end{theorem}
\

We define the moduli space $\mathbf{\overline{M}}_{1,0}^{ch}(X,d)$ of admissible stable maps to chain type FM spaces of $X$ to be same as $\mathbf{\overline{M}}_{1,0}^{log,ch}(X,d)$ without log structures. That is, an element of $\mathbf{\overline{M}}_{1,0}^{ch}(X,d)$ is an element of $\mathbf{\overline{M}}_{1,0}^{log,ch}(X,d)$ without log structures.

\bigskip
\bigskip
\bigskip
\bigskip

%%%%%%%%%%%%%%%%%%%%%%%%%%%%%%%%%%%%%%%%%%%%%%%%%%%%%%

\section{An example of degeneration}

\bigskip

First we construct a family of elliptic stable maps over $S=\mathbb{A}^{2}$.
Let $R=k[t,a]$ be a coordinate ring of $S$, where k is a algebraically closed field and $t$,$a$ are indeterminates.
Let $C'=Proj(R[x,y,z]/zy^{2}-x^{3}-z^{2}x-z^{3})$. Let $f':C'\dashrightarrow\mathbb{P}^{2}$ be given by $[t^{3}y,at^{2}x,z]$. It is a well defined family of elliptic stable maps except at $\{ (t,a) : t=0\} \subset S$. If we blow up an ideal $(t,x,z)$, it extends to family of elliptic stable maps on whole $S$. That is, if we let $C=Bl_{(t,x,z)}C'$, the rational morphism $f':C'\dashrightarrow\mathbb{P}^{2}$ extends to $f:C\longrightarrow\mathbb{P}^{2}$ and $f$ gives a family of elliptic stable maps over $S$. At $t\neq0$, its domain curve is smooth. At $t=0$ its domain curve consists of an elliptic component whose degree is $0$ and one rational component whose morphism is given by $[s^{3},as^{2},1]$ where $s$ is the local coordinate of the rational component such that $\{s=0\}$ is the intersection point with elliptic component.\\

Now we construct a family of elliptic stable admissible maps over $\widetilde{S}=Bl_{(t,a)}S$ in following way. 

 \begin{proposition}
 Let $R$, $C'$, $C$, $f$ be as above. Let $\widetilde{S}$ be the blow up of $S$ at the origin and let $E$ be the exceptional divisor. Let $D$ be the proper transform of subscheme defined by $(t)\subset R$.
 Let $C''$ be a pullback of $C'$ along $\widetilde{S}\longrightarrow S$ and $\widetilde{C}$ be the blow up of $C''$ along ideals $(D,x,z)$, $(E,x,z)$; Here we mean that first blow up along $(D,x,z)$ and next blow up along the proper transform of $(E,x,z)$. Let $\widetilde{W}$ be the blow up of $\widetilde{S}\times\mathbb{P}^{2}$ along ideals $(E^{3},x_{0},x_{1}), (D^{2},x_{0},x_{1})$, where $x_{0}$, $x_{1}$, $x_{2}$ are coordinates of $\mathbb{P}^{2}$. Then $f:C\longrightarrow\mathbb{P}^{2}$ extends to map $\widetilde{f}:\widetilde{C}\longrightarrow\widetilde{W}$ and $(\widetilde{f}:\widetilde{C}\longrightarrow\widetilde{W})$ is a family of admissible stable maps.
 \end{proposition}
    \begin{proof}
        we choose one local coordinate of $\widetilde{S}$ as $\{(t,a)\}\simeq\mathbb{A}^{2}$ such that $\widetilde{S}\longrightarrow S$ is given by $(t,a)\mapsto (ta,a)$. Then the induced morphism is given by $[t^{3}a^{3}y:t^{2}a_{3}x,z]$. Since we only need to consider a neighborhood of $\{[x:y:z]=[0,1,0]\}$ which is smooth point, the problem is reduced to the following lemma.
        \begin{lemma} Suppose $(f:C=\mathbb{A}^{1}\times\mathbb{A}^{2}=Spec(k[x,t,a])\longrightarrow W=\mathbb{A}^{2}\times\mathbb{A}^{2}=Spec(k[X,Y,t,a]))$ is given by $(x,t,a)\mapsto(\frac{t^{3}a^{3}}{z},\frac{t^{2}a^{3}x}{z},t,a)$, where $z$ is a function of $x$ such that the vanishing order of z at $x=0$ is 3. If we let $\widetilde{C}$ be the blow up of $C$ along ideals $(x,t)$, $(x,a)$ and $\widetilde{W}$ be the blow up of $W$ along ideals $(X,Y,a^{3})$, $(X,Y,t^{2})$, then $(f:C\longrightarrow W)$ extends to morphism $(\widetilde{f}:\widetilde{C}\longrightarrow\widetilde{W})$
        \end{lemma}
           \begin{proof}
              Using universal property of blow ups, we need to check that inverse image sheaves of $(X,Y,a^{3})$ and $(X,Y,t^{3})$ are invertible. 
              For example, at an open set $U\subset C$ given by $\{(x,t,a):x\neq 0)\}$, inverse image sheaves of $(X,Y,a^{3})$ and $(X,Y,t^{3})$ are $(\frac{t^{3}a^{3}}{z},\frac{t^{2}a^{3}x}{z},a_{3})=(a^{3})$ and $(\frac{t^{3}a^{3}}{z},\frac{t^{2}a^{3}x}{z},t^{2})=(t^{2})$ respectively which are invertible sheaves.
              Other cases are left to the reader.          
           \end{proof}

We can easily check that $\widetilde{f}$ satisfies admissible conditions.
In the same way we can prove the case of the other open sets of $\widetilde{S}$. This proves theorem.        
    \end{proof}

\remark the origin in $S$ parameterizes stable map whose domain curve consist of an elliptic component of degree $0$ and one rational component whose morphism has ramification order $3$ at the intersection point with the elliptic component. i.e. it is an element of $\sum_{1}$. 
\remark In the proof, we can describe an element of admissible stable map explicitly. For example over $\{(t,a):a=0, t\neq 0\}\subset\widetilde{S}$, $\widetilde{C}$ is of the type $E[C_{1}]$ and $\widetilde{W}=\mathbb{P}^{2}(1)$ and $\widetilde{f}|_{E}:E\longrightarrow\mathbb{P}^{2}_{1}$ is given by $[X_{0},X_{1},X_{2}]=[t^{3}y:t^{2}x:z]=[ty:x:z]$ where $X_{0}, X_{1}, X_{2}$ are coordinates of $\mathbb{P}^{2}_{1}$ such that $D(1)$ is given by $\{[X_{1},X_{2},X_{3}]:X_{2}=0\}$. The last equality is due to the existence of an automorphism of $\mathbb{P}^{2}_{1}$ fixing $D(1)$.\
\\
\\
\\

%%%%%%%%%%%%%%%%%%%%%%%%%%%%%%%%%%%%%%%%%%%%%%%%%%%
%%%%%%%%%%%%%%%%%%%%%%%%%%%%%%%%%%%%%%%%%%%%%%%%%
\section{The description of fiber in the moduli space of elliptic admissible stable maps}
\bigskip
By the definition of admissible stable map we get the following proposition.

\begin{proposition}There is a natural morphism $\phi$ from the moduli space of elliptic admissible stable maps to the Kontsevich's moduli space of stable maps.
\begin{proof}
A family of admissible maps over $\tilde{S}$ consist of $((\tilde{C}/\tilde{S}, \mathbf{p}), W/\tilde{S}, \tilde{f}: \tilde{C} \longrightarrow W)$, where $\tilde{C}$ is a family of pre-stable curves over $\tilde{S}$ and $W$ is $FM$ type space of $\mathbb{P}^{2}$. By just forgetting W, we obtain Kontsevich's pre-stable maps and after stabilization we get Kontsevich's stable maps.
\end{proof}
\end{proposition}

Now we describe set theoretic fibers of $\phi$, when $d=3$. Note that if the essential part is not contracted to a point, the fiber is just one point because $W$ is trivial. i.e. $W =\mathbb{P}^{n}$.
Let's consider the fiber of element where the essential part is contracted to a point.

\lemma{} Let $(C,f:C\longrightarrow \mathbb{P}^{n})$ be an element of the main component of the Kontsevich's moduli space of elliptic stable maps satisfying following condition.
$C$ is of the type $E(0)[C_{1}(3)]$. 
\begin{enumerate}
  \item if $f$ has ramification order $2$ at $c_{1}$ and there is no smooth point $q_{1}\in C_{1}$ such that $f(c_{1})=f(q_{1})$,
then the fiber of $\phi$ is equal to a point set theoretically.
  \item if $f$ has ramification order $2$ at $c_{1}$ and there is a smooth point $q_{1}\in C_{1}$ such that $f(c_{1})=f(q_{1})$,
then the fiber of $\phi$ is equal to $\mathbb{P}^{n-1}$ set theoretically.
  \item if $f$ has ramification order $3$ at $c_{1}$,
then the fiber of $\phi$ is equal to $Bl_{pt}\mathbb{P}^{n}$ set theoretically.
\end{enumerate}

\begin{proof}

 \begin{enumerate}
  \item \begin{itemize}
            \item \textbf{point : } The domain curve $\tilde{C}$ is of the type $E[C_{1}]$ and $W=\mathbb{P}^{n}(1)$. $\tilde{f}:C_{1}\longrightarrow\mathbb{P}^{n}_{0}$ is already given. $\tilde{f}:E\longrightarrow\mathbb{P}^{n}_{1}$ is given by $[X_{0}:X_{1}:\cdots:X_{n}]=[x:0:\cdots:0:z]$, where $E$ are given by $\{[x,y,z]:zy^{2}=x^{3}+z^{2}x+Az^{3}\}$ and $X_{0}$, $X_{1}$, $\cdots$, $X_{n}$ are coordinates of $\mathbb{P}^{n}$. $c_{1}$ is given by $\{[x:y:z]: z=x=0\}$ and $D(1)$ is given by $\{[X_{1},X_{2},\cdots,X_{n}] : X_{n}=0\}$.          
           \end{itemize}\
  
  \item \begin{itemize} 
            \item \textbf{ $\mathbb{A}^{n-1}$ with parameter \{ [$\alpha_{0} :\alpha_{1}:\cdots:\alpha_{n-1}$] , $\alpha_{n-1}\neq 0$ \} : }  The domain curve $\tilde{C}$ is of the type $E[C_{1}[A_{1}]]$ and $W=\mathbb{P}^{n}(1)$. $\tilde{f}|_{C_{1}}:C_{1}\longrightarrow \mathbb{P}^{n}_{0}$ is already given. $\tilde{f}|_{E}:E\longrightarrow\mathbb{P}^{n}_{1}$ is given by $[X_{0}:X_{1}:\cdots:X_{n}]=[x:0:\cdots:0:z]$. $\tilde{f}|_{A_{1}}:A_{1}\longrightarrow\mathbb{P}^{n}(1)$ is given by $[1:\alpha_{0}t :\alpha_{1}t:\cdots:\alpha_{n-1}t]$, where $t$ are a local parameter of $A_{1}$ such that $a_{1}$ is given by $\{t=0\}$ and $D(1)$ is given by $\{[X_{1},X_{2},\cdots,X_{n}] : X_{n}=0\}$.\
            \item \textbf{$\mathbb{P}^{n-2}$ with parameter \{ [$\alpha_{0} :\alpha_{1}:\cdots:\alpha_{n-2}$] \} : } The domain curve $\tilde{C}$ is of the type $E[C_{1}'[C_{1}[A_{1}]]]$ and $W=\mathbb{P}^{n}(2)$. $\tilde{f}|_{C_{1}}:C_{1}\longrightarrow\mathbb{P}^{n}_{0}$ already given. $\tilde{f}|_{C_{1}'}:C_{1}'\longrightarrow\mathbb{P}^{n}_{1}$ is given by $[t^{2}:0:\cdots:0:1]$ where $t$ is a local parameter of $C_{1}'$ such that $c_{1}$ are given by $\{t=0\}$ and $D(1)$ is given by $\{[X_{1},X_{2},\cdots,X_{n}] : X_{n}=0\}$. $\tilde{f}|_{E}:E\longrightarrow\mathbb{P}^{n}_{2}$ is given by $[x:0:\cdots:0:z]$ where $D(2)$ are given by $\{[X_{1},X_{2},\cdots,X_{n}] : X_{n}=0\}$. $\tilde{f}|_{A_{1}}:A_{1}\longrightarrow\mathbb{P}^{2}_{1}$ is given by $[1:\alpha_{0}s:\alpha_{1}s:\cdots:\alpha_{n-2}s:s]$ where $s$ is a local parameter of $A_{1}$ such that $a_{1}$ is given by $\{s=0\}$.
          \end{itemize}\
  
  \item \begin{itemize} 
            \item \textbf{$\mathbb{A}^{n}$ with parameter \{$[\alpha_{0}:\alpha_{1}:\cdots:\alpha_{n}]$,$\alpha_{n}\neq 0$\} : } The domain curve $\tilde{C}$ is of the type $E[C_{1}]$ and $W=\mathbb{P}^{n}(1)$. $\tilde{f}|_{C_{1}}:C_{1}\longrightarrow\mathbb{P}^{n}_{0}$ already given. $\tilde{f}|_{E}:E\longrightarrow\mathbb{P}^{n}_{1}$ is given by $[X_{0}:X_{1}:\cdots:X_{n}]=[\alpha_{0} x +\alpha_{n}y:\alpha_{1} x:\alpha_{2}:\cdots:\alpha_{n-1}:z]$, where $D(1)$ is given by $\{[X_{1},X_{2},\cdots,X_{n}] : X_{n}=0\}$.\\
            
            \item \textbf{$\mathbb{P}^{n-1}$ \textbackslash  $pt$ with parameter \{$[\alpha_{0}:\alpha_{1}:\cdots:\alpha_{n-1}]$, not all $\alpha_{k}$ are 0 for $1\leqslant k\leqslant n-1$\} : } The domain curve $\tilde{C}$ is of the type $E[C_{1}'[C_{1}]]$ and $W=\mathbb{P}^{n}(2)$. $\tilde{f}|_{C_{1}}:C_{1}\longrightarrow\mathbb{P}^{n}_{0}$ is already given. $\tilde{f}|_{C_{1}'}:C_{1}'\longrightarrow\mathbb{P}^{n}_{1}$ is given by $[1-\alpha_{0}t:\alpha_{1}t:\alpha_{2}t:\cdots:\alpha_{n-1}t:t^{3}]$ where $t$ is a local parameter of $C_{1}'$ such that $c_{1}$ is given by $\{t=0\}$ and $D(1)$ is given by $\{[X_{1},X_{2},\cdots,X_{n}] : X_{n}=0\}$. $\tilde{f}|_{E}:E\longrightarrow\mathbb{P}^{n}_{2}$ is given by $[x:0:\cdots:0:z]$ where $D(2)$ are given by $\{[X_{1},X_{2},\cdots,X_{n}] : X_{n}=0\}$.\\
            
            \item \textbf{$\mathbb{A}^{n-1}$ with parameter \{ [$\alpha_{0} :\alpha_{1}:\cdots:\alpha_{n-1}$] , $\alpha_{n-1}\neq 0$ \} : } The domain curve $\tilde{C}$ is of the type $E[C_{1}'[C_{1},A_{1}]]$ and $W=\mathbb{P}^{n}_{2}$. $\tilde{f}|_{C_{1}}:C_{1}\longrightarrow\mathbb{P}^{n}_{0}$ is already given. $\tilde{f}|_{C_{1}'}:C_{1}'\longrightarrow\mathbb{P}^{n}_{1}$ is given by $[1-t:0:\cdots:0:t^{3}]$ where $t$ is a local parameter of $C_{1}'$ such that $c_{1}$ is given by $\{t=0\}$ and $a_{1}$ is given by $\{t=1\}$ and $D(1)$ is given by $\{[X_{1},X_{2},\cdots,X_{n}] : X_{n}=0\}$. $\tilde{f}|_{E}:E\longrightarrow\mathbb{P}^{n}_{2}$ is given by $[x:0:\cdots:0:z]$. $\tilde{f}|_{A_{1}}:A_{1}\longrightarrow\mathbb{P}^{n}_{2}$ is given by $[1:\alpha_{0}:\alpha_{1}:\cdots:\alpha_{n-1}s]$, where $s$ is a local parameter of $A_{1}$ such that $a_{1}$ is given by $\{s=0\}$ and $D(2)$ is given by $\{[X_{1},X_{2},\cdots,X_{n}] : X_{n}=0\}$.\\
            
            \item \textbf{$\mathbb{P}^{n-2}$ with parameter \{ [$\alpha_{0} :\alpha_{1}:\cdots:\alpha_{n-2}$] \} : } The domain curve $\tilde{C}$ is of the type $E[C_{1}''[C_{1}'[C_{1},A_{1}]]]$ and $W=\mathbb{P}^{n}(3)$. $\tilde{f}|_{C_{1}}:C_{1}\longrightarrow\mathbb{P}^{n}_{0}$ is already given. $\tilde{f}|_{C_{1}'}:C_{1}'\longrightarrow\mathbb{P}^{n}_{1}$ is given by $[1-t:0:\cdots:0:t^{3}]$ where $t$ is a local parameter of $C_{1}'$ such that $c_{1}$ is given by $\{t=0\}$ and $a_{1}$ is given by $\{t=1\}$ and $D(1)$ is given by $\{[X_{1},X_{2},\cdots,X_{n}] : X_{n}=0\}$. $\tilde{f}|_{C_{1}''}:C_{1}''\longrightarrow\mathbb{P}^{n}_{2}$ is given by $[1:0:\cdots:0:s^{2}]$ where $s$ is parameter of $C_{1}''$ such that $c_{1}'$ is given by $\{s=0\}$. $\tilde{f}|_{A_{1}}:A_{1}\longrightarrow\mathbb{P}^{n}_{2}$ is given by $[1:\alpha_{0}u:\alpha_{1}u:\cdots:\alpha_{n-2}u:u]$ where $u$ is a local parameter of $A_{1}$ such that $a_{1}$ is given by $\{u=0\}$ and $D(2)$ is given by $\{[X_{1},X_{2},\cdots,X_{n}] : X_{n}=0\}$. $\tilde{f}|_{E}:E\longrightarrow\mathbb{P}^{n}_{3}$ is given by $[x:0:\cdots:0:z]$ where $D(3)$ is given by $\{[X_{1},X_{2},\cdots,X_{n}] : X_{n}=0\}$.
             \end{itemize}
   \end{enumerate} 
\end{proof}

Note that the case where essential part is singular curves can be stated and proved in the same way.
Note that we actually know every element of the fiber explicitly. By similar way we can prove the following lemmas whose proof will be omitted. \\

\lemma{} Let $(C,f:C\longrightarrow \mathbb{P}^{n})$ be an element of the main component of the Kontsevich's moduli space of stable maps satisfying the following conditions.
$C$ is of the type $E(0)[C_{1}(1),C_{2}(2)]$.
\begin{enumerate}
  \item if $f$ has ramification order $1$ at $c_{2}$ and there is no smooth point $q_{2}\in C_{2}$ such that $f(c_{2})=f(q_{2})$,
then the fiber of $\phi$ is equal to a point set theoretically.
  \item if $f$ has ramification order $1$ at $c_{2}$ and there is a smooth point $q_{2}\in C_{2}$ such that $f(c_{2})=f(q_{2})$, then the fiber of $\phi$ is equal to $\mathbb{P}^{n}$ set theoretically.
  \item if $f$ has ramification order $2$ at $c_{2}$.and images of $C_{1}$ and $C_{2}$ are different lines, then fiber of $\phi$ is equal to $\mathbb{P}^{1}$ set theoretically.
 \item if $f$ has ramification order $2$ at $c_{2}$, images of $C_{1}$ and $C_{2}$ are same lines, then the fiber of $\phi$ is equal to $\mathbb{P}^{n-1}\bigcup\mathbb{P}^{1}$ glued at one point, set theoretically.
\end{enumerate}\

\lemma{} Let $(C,f:C\longrightarrow \mathbb{P}^{n})$ be an element of the main component of the Kontsevich's moduli space of stable maps satisfying the following conditions.
$C$ is of the type $E(0)[C_{1}(1),C_{2}(1),C_{3}(1)]$.
\begin{enumerate}
  \item if images of $C_{1}$ and $C_{2}$ and $C_{3}$ are distinct lines, then the fiber of $\phi$ is equal to a point set theoretically.
  \item if images of $C_{1}$ and $C_{2}$ are same lines and the image of $C_{3}$ is the distinct line, then the fiber of $\phi$ is equal to a point set theoretically.
  \item if images of $C_{1}$ and $C_{2}$ and $C_{3}$ are all same lines, then the fiber of $\phi$ is equal to $\mathbb{P}^{1}$ set theoretically.
\end{enumerate}\

\lemma{} Let $(C,f:C\longrightarrow \mathbb{P}^{n})$ be an element of the main component of the Kontsevich's moduli space of stable maps satisfying the following conditions.
$C$ is of the type $E(0)[B_{1}(0)[C_{1}(1),C_{2}(2)]]$.
\begin{enumerate}
  \item if $f$ has ramification order $1$ at $c_{2}$ and there is no smooth point $q_{2}\in C_{2}$ such that $f(c_{2})=f(q_{2})$, then the fiber of $\phi$ is equal to a point, set theoretically.
  \item if $f$ has ramification order $1$ at $c_{2}$ and there is a smooth point $q_{2}\in C_{2}$ such that $f(c_{2})=f(q_{2})$, then the fiber of $\phi$ is equal to $\mathbb{P}^{n-1}\bigcup\mathbb{P}^{n-1}$ glued along $\mathbb{P}^{n-2}$, set theoretically.
  \item if $f$ has ramification order $2$ at $c_{2}$ and the tangent lines of images of $C_{1}$ and $C_{2}$ are independent, then the fiber of $\phi$ is equal to $\mathbb{P}^{1}$, set theoretically.
  \item if the tangent lines of images of $C_{1}$ and $C_{2}$ are dependent, then the fiber of $\phi$ is equal to $Bl_{pt}\mathbb{P}^{n}\bigcup (\mathbb{P}^{1}\times\mathbb{P}^{n-1}) \bigcup Bl_{pt}\mathbb{P}^{n}$ glued along $\mathbb{P}^{n-1}$, $\mathbb{P}^{n-2}$, set theoretically.
\end{enumerate}\

\lemma{} Let $(C,f:C\longrightarrow \mathbb{P}^{n})$ be an element of the main component of the Kontsevich's moduli space of stable maps satisfying the following conditions.
$C$ is of the type $E(0)[B_{1}(0)[C_{1}(1),C_{2}(1),C_{3}(1)]]$.
\begin{enumerate}
  \item if the images of $C_{1}$ and $C_{2}$ and $C_{3}$ are distinct lines, then the fiber of $\phi$ is equal to a point set theoretically.
  \item if the images of $C_{1}$ and $C_{2}$ are same line and the image of $C_{3}$ is distinct line, then the fiber of $\phi$ is equal to a point set theoretically.
  \item if the images of $C_{1}$ and $C_{2}$ and $C_{3}$ are all same lines, then the fiber of $\phi$ is equal to $(\mathbb{P}^{1}\times\mathbb{P}^{n-1})\bigcup Bl_{pt}\mathbb{P}^{n}$ glued along $\mathbb{P}^{n-1}$set theoretically.
\end{enumerate}\

\lemma{} Let $(C,f:C\longrightarrow \mathbb{P}^{n})$ be an element of the main component of the Kontsevich's moduli space of stable maps satisfying the following conditions.
$C$ is of the type $E(0)[B_{1}(0)[C_{1}(1),C_{2}(1)],C_{3}(1)]]$.
\begin{enumerate}
  \item if the images of $C_{1}$ and $C_{2}$ are distinct lines, then fiber of $\phi$ is equal to a point, set theoretically.
  \item if the images of $C_{1}$ and $C_{2}$ are same lines and the image of $C_{3}$ is distinct line, then the fiber of $\phi$ is equal to $\mathbb{P}^{n-1}$, set theoretically.
  \item if the images of $C_{1}$ and $C_{2}$ and $C_{3}$ are all same lines, then the fiber of $\phi$ is equal to $(\mathbb{P}^{1}\times\mathbb{P}^{n-1})\bigcup\mathbb{P}^{n-1}$ glued along $\mathbb{P}^{n-2}$, set theoretically.
\end{enumerate}\

\lemma{} Let $(C,f:C\longrightarrow \mathbb{P}^{n})$ be an element of the main component of the Kontsevich's moduli space of stable maps satisfying the following conditions.
$C$ is of the type $E(0)[B_{1}(0)[B_{2}(0)[C_{1}(1),C_{2}(1)],C_{3}(1)]]$.
\begin{enumerate}
  \item if the images of $C_{1}$ and $C_{2}$ are distinct lines, then the fiber of $\phi$ is equal to point, set theoretically.
  \item if the images of $C_{1}$ and $C_{2}$ are same lines and the image of $C_{3}$ is the distinct line, then the fiber of $\phi$ is equal to $\mathbb{P}^{n-1}$, set theoretically.
  \item if the images of $C_{1}$ and $C_{2}$ and $C_{3}$ are all same lines, then the fiber of $\phi$ is equal to $ Bl_{pt}\mathbb{P}^{n}\bigcup Bl_{pt}(\mathbb{P}^{n-1}\times\mathbb{P}^{1})\bigcup (\mathbb{P}^{n-1}\times\mathbb{P}^{1})\bigcup(\mathbb{P}^{n-1}\times\mathbb{P}^{1})$, set theoretically.
  
 \end{enumerate}\

\remark What we showed is that the fiber of $\phi$ is ,at least set theoretically, same as the fiber of corresponding blow-ups which we will describe later. Actually it is same scheme theoretically.\\
\\
\\
\\
%%%%%%%%%%%%%%%%%%%%%%%%%%%%%%%%%%%%%%%%%%%%%%%%%%%%%
\section{The case of the degree $2$}
\bigskip

In this section, we show that when $d=2$, two moduli spaces are same. i.e. $\mathbf{\widetilde{M}}_{1,0}(\mathbb{P}^{n},2)_{0}$ =  $\overline{\mathbf{M}}_{1,0}^{log,ch}(\mathbb{P}^{n},2)$.
Note that if the degree is $2$, $\mathbf{\widetilde{M}}_{1,0}(\mathbb{P}^{n},2)_{0} = \mathbf{\overline{M}}_{1,0}(\mathbb{P}^{n},2)_{0}$.
As in the previous section, we can calculate the fiber of $\phi:\overline{\mathbf{M}}_{1,0}^{ch}(\mathbb{P}^{n},2)\longrightarrow\mathbf{\widetilde{M}}_{1,0}(\mathbb{P}^{n},2)_{0}$ and it is easy to see that every fiber is just one point. This actually suffices to conclude that $\mathbf{\widetilde{M}}_{1,0}(\mathbb{P}^{n},2)_{0}$ =  $\overline{\mathbf{M}}_{1,0}^{log,ch}(\mathbb{P}^{n},2)$ by the Zariski's main theorem. Still we construct an actual morphism for the completeness. We only do the case $n=1$ for simplicity.

Note that when the essential part is not contracted to a point, two moduli spaces are naturally isomorphic. So we only need to consider neighborhoods of points where the essential part is contracted to point.

We describe an etale atlas of stack $\mathbf{\overline{M}}_{1,0}(\mathbb{P}^{1},2)_{0}$. Because of stackyness of the moduli space of elliptic curves, we need to separate the case according to j-invariant of the essential part of the domain curve.
\subsubsection{when essential part is smooth elliptic curve with j $\neq0$}
Let $k$ be an algebraically closed field and $t$, $\alpha$, $\gamma$, $c$, $A$ be indeterminates.
Let $R=k[t,\alpha,\gamma,c,A]/(\gamma-\alpha^{3}-\gamma^{2}\alpha-A\gamma^{3})$.
Let $D_{1}$,$D_{2}$ be subschemes defined by ideals $(\alpha,\gamma)$,$(t)$.
Let $S=Spec(R)$\textbackslash$V$ where $V\subset Spec(R)$ is a subscheme defined by an ideal $(4+27A^{2})$ and $C'=Proj(R[x,y,z]/zy^{2}-x^{3}-z^{2}x-Az^{3})$.\\

Then the rational map $f':C'\dashrightarrow \mathbb{P}^{1}$ defined by $[t\gamma(x+\alpha y)+c(\gamma x-\alpha z),\gamma x-\alpha z]$ gives the family of elliptic stable maps except at $D_{1}$ and $D_{2}$.
But if we let $C$ be the blow up of $C'$ along $(D_{1},x,z)$,$(D_{2},x-\alpha y,z-\gamma y)$,$(D_{2},x,z)$, we easily see that $f':C'\dashrightarrow \mathbb{P}^{1}$ extends to $f:C\longrightarrow \mathbb{P}^{1}$ and $f$ gives a family of elliptic stable maps over whole $S$.\\

Moreover we know every element of family over $S$ explicitly as follows. \\
Over \{$\gamma=0,t\neq0$\}, the domain curves are of the type $E[C_{1}]$ and $f|_{C_{1}}:C_{1}\rightarrow\mathbb{P}^{1}$ is given by $[ts^{2}+c(s-1):s-1]$, where $s$ is a local parameter of $C_{1}$ such that $c_{1}$ is given by $\{s=0\}$.\\
Over \{$t=0,\gamma\neq0$\}, the domain curves are of the type $E[C_{1},C_{2}]$ and $c_{1}$ and $c_{2}$ are given by $(z=0)$, $(x=\alpha y,z=\gamma y)$ in $E=\{[x;y;z]:zy^{2}=x^{3}+z^{2}x+Az^{3}\}$ and $f|_{C_{i}}:C_{i}\rightarrow\mathbb{P}^{1}$ is given by $[s_{i}+c,1]$ where $c_{i}$ is a local parameter of $C_{i}$ such that $c_{i}$ is given by $(s_{i}=0)$ for $i=1, 2$.\\
Over \{$\gamma=0,t=0$\}, the domain curves are of the type $E[B_{1}[C_{1},C_{2}]]$ and $f|_{C_{i}}:C_{i}\rightarrow\mathbb{P}^{1}$ is given by $[s_{i}+c,1]$ where $s_{i}$ local parameter of $C_{i}$ such that $c_{i}$ is given by $(s_{i}=0)$ for $i=1, 2$.\\
By looking at a local deformation, we can check that $S$ is an etale atlas of $\mathbf{\overline{M}}_{1,0}(\mathbb{P}^{1},2)_{0}$.

 \begin{proposition}
 Let us assume above. If we let $W$ be the blow up of $S\times \mathbb{P}^{1}$ along $(D_{2},x_{0}-cx_{1})$, $(D_{1}^{2},x_{0}-cx_{1})$ where $x_{0}$, $x_{1}$ are coordinates of $\mathbb{P}^{1}$, then $f:C\longrightarrow \mathbb{P}^{1}$ extends to $\widetilde{f}:C\longrightarrow W$ and $\widetilde{f}$ gives a family of elliptic admissible maps over $S$.
 \end{proposition}

\remark We know every element over $S$ explicitly. Here we only describe the type of domain curves and target. Over $\{t\neq0,\gamma\neq=0\}$, the domain curve $C$ is of type $E$  and the target $W=\mathbb{P}^{1}$.
Over $\{t=0,\gamma\neq 0\}$, the domain curve $C$ is of the type $E[C_{1},C_{2}]$ and the target $W=\mathbb{P}^{1}(1)$.
Over $\{t\neq 0,\gamma =0\}$, the domain curve $C$ is of the type $E[C_{1}]$ and the target $W=\mathbb{P}^{1}(1)$.  
Over $\{t=0,\gamma =0\}$, domain $C=E[B_{1}[C_{1},C_{2}]]$ and the target $W=\mathbb{P}^{1}(2)$.

\subsubsection{when essential part is smooth elliptic curve with j $\neq1728$} Everything is same if we change the equation $\gamma-\alpha^{3}-\gamma^{2}\alpha-A\gamma^{3}$ above to $\gamma-\alpha^{3}-A\gamma^{2}\alpha-\gamma^{3}$.

\subsubsection{when essential part is singular curve}
Let $k$ be an algebraically closed field and $t$, $\alpha$, $\beta$, $c$, $A$ be indeterminates.
Let $R=k[t,\alpha,\beta,c,A]/(\beta^{2}-\alpha^{3}-\alpha^{2}-A)$.
Let $D_{2}$ be a subscheme defined by an ideal $(t)$.
Let $S=Spec(R)$ and $C'=Proj(R[x,y,z]/(zy^{2}-x^{3}-zx^{2}-Az^{3}))$.

Then a rational map $f':C'\dashrightarrow \mathbb{P}^{1}$ defined by $[t(y+\beta z)+c(x-\alpha z),x-\alpha z]$ gives a family of elliptic stable maps except at $D_{2}$ and $\{\alpha=\beta=0\}$.
But if we let $C$ be the blow up of $C'$ along ideals $(y-x-\frac{x^{2}+\alpha x+\alpha^{2}}{2},\beta-\alpha-\frac{x^{2}+\alpha x+\alpha^{2}}{2})$, $(D_{2},x,z)$, $(D_{2},x-\alpha z,y-\beta z)$, we can see that $f':C'\dashrightarrow\mathbb{P}^{1}$ extends to $f:C\longrightarrow\mathbb{P}^{1}$ and $f$ gives a family of elliptic stable maps over whole $S$. Note that the effect of blowing up along ideal $(y-x-\frac{x^{2}+\alpha x+\alpha^{2}}{2},\beta-\alpha-\frac{x^{2}+\alpha x+\alpha^{2}}{2})$ is inserting a rational component at singular point in the rational nodal curve at $\{\alpha=\beta=0\}$.

 \begin{proposition}
 Let us assume above. If we let $W$ be the blow up of $S\times \mathbb{P}^{1}$ along ideal $(D_{2},x_{0}-cx_{1})$ where $x_{0}$,$x_{1}$ are coordinates of $\mathbb{P}^{1}$, then $f:C\longrightarrow\mathbb{P}^{1}$ extends to $\widetilde{f}:C\longrightarrow W$ and $\widetilde{f}$ gives a family of elliptic admissible maps over $S$.
 \end{proposition}

\subsubsection{}
Summing up previous results we get a morphism from $\mathbf{\widetilde{M}}_{1,0}(\mathbb{P}^{1},2)_{0}$ to $\overline{\mathbf{M}}_{1,0}^{ch}(\mathbb{P}^{1},2)$ where $\overline{\mathbf{M}}_{1,0}^{ch}(\mathbb{P}^{1},2)$ is the moduli space of elliptic admissible stable maps without log structures. Actually it is one to one morphism. On the other hand we also have one to one morphism from $\overline{\mathbf{M}}_{1,0}^{ch,log}(\mathbb{P}^{1},2)$ to $\overline{\mathbf{M}}_{1,0}^{ch}(\mathbb{P}^{1},2)$, which is just forgetting log structures(\cite{Kim}). By the uniqueness of the normalization, we see that $\mathbf{\widetilde{M}}_{1,0}(\mathbb{P}^{1},2)_{0} = \overline{\mathbf{M}}_{1,0}^{ch,log}(\mathbb{P}^{1},2)$.\

\bigskip
\bigskip
\bigskip
%%%%%%%%%%%%%%%%%%%%%%%%%%%%%%%%%%%%%%%%%%%%%%%%%%%%%%%%%%%
\section{The case of the degree $3$}
\bigskip
In this section we describe a local chart of $\mathbf{\overline{M}}_{1,0}(\mathbb{P}^{n},3)_{0}$ and $\overline{\mathbf{M}}_{1,0}^{log,ch}(\mathbb{P}^{n},3)$. Throughout the section, we will denote the proper transforms of subscheme as the same notations as original subschemes.\\

 Because of the stackyness of the moduli space of elliptic curves, we need to separate the case according to j-invariant of the essential part of the domain curve.
\subsection{Etale chart of $\mathbf{\overline{M}}_{1,0}(\mathbb{P}^{n},3)_{0}$}

(a)when essential part is smooth elliptic curve with j $\neq0$.\\
Let $k$ be an algebraically closed field and $a_{1}$, $a_{2}$, $\cdots$, $a_{n-1}$, $b_{1}$, $b_{2}$, $\cdots$, $b_{n-1}$, $c_{1}$, $c_{2}$, $\cdots$, $c_{n-1}$, $d_{1}$, $d_{2}$, $\cdots$, $d_{n}$, $z_{1}$, $z_{2}$, $A$, $\alpha$, $\gamma$, $\alpha'$, $\gamma'$ be indeterminates.
Let $R$ = k[$a_{1}$,$a_{2}$,$\cdots$,$a_{n-1}$,$b_{1}$,$b_{2}$,$\cdots$,$b_{n-1}$,$c_{1}$,$c_{2}$,$\cdots$,$c_{n-1}$,$d_{1}$,$d_{2}$,$\cdots$,$d_{n}$,$z_{1}$,$z_{2}$,$A$,$\alpha$,$\gamma$,$\alpha'$,$\gamma'$]/($a_{1}z_{1}-b_{1}z_{2}$,$a_{2}z_{1}-b_{2}z_{2}$,$\cdots$,$a_{n-1}z_{1}-b_{n-1}z_{2}$,$\gamma-\alpha^{3}-\gamma^{2}\alpha-A\gamma^{3}$,$\gamma'-\alpha'^{3}-\gamma'^{2}\alpha'-A\gamma'^{3}$).\

Let $D_{2,\alpha}$, $D_{2,\alpha'}$, $D_{2,\alpha-\alpha'}$, $D_{3}$, $F_{\alpha}$, $F_{\alpha'}$, $F_{\alpha-\alpha'}$, $G$ be subschemes defined by ideals $(\alpha,\gamma)$,
$(\alpha',\gamma')$, $(\alpha-\alpha',\gamma-\gamma')$, $(z_{1},z_{2})$, $(z_{1},b_{1},b_{2},\cdots,b_{n-1})$, $(z_{2},a_{1},a_{2},\cdots,a_{n-1})$, $(z_{1}-z_{2},a_{1}-b_{1},a_{2}-b_{2},\cdots,a_{n-1}-b_{n-1})$, $(\frac{\alpha'\gamma \frac{z_{1}}{z_{2}}(x+\alpha y)(\gamma'x-\alpha'z)-\alpha\gamma'(\gamma x-\alpha z)(x+\alpha'y)}{\alpha \alpha'(\alpha-\alpha')(\gamma x-\alpha z)(\gamma'x-\alpha'z)})$.

Let $\widehat{S}=Spec(R)$ \textbackslash $V$, $S=Bl_{(\alpha,\alpha',\gamma,\gamma')}\widehat{S}$, where $V$ is the subscheme defined by an ideal $(4+27A^{2})(\alpha-\alpha',\gamma-\gamma')$. Let $D_{1}$ be the exceptional divisor.
Let $\widehat{C}=Proj(R[x,y,z]/zy^{2}-x^{3}-z^{2}x-Az^{3})$, $C'$ be a pull-back of $\widehat{C}$, and $C$ be the blow up of $C'$ along ideals \\

$(D_{1},x,z)$,

$(D_{2,\alpha},x-\alpha'y,z-\gamma'y)$,$(D_{2,\alpha},x,z)$,$(D_{2,\alpha'},x-\alpha y,z-\gamma y)$,$(D_{2,\alpha'},x,z)$,

$(D_{3},x,z)$,$(D_{3},x-\alpha y,z-\gamma y)$,$(D_{3},x-\alpha'y,z-\gamma'y)$,

$(F_{\alpha},x-\alpha y,z-\gamma y)$,$(F_{\alpha'},x-\alpha'y,z-\gamma'y)$,$(F_{\alpha-\alpha'},x,z)$.\\

Then rational map $\widehat{f}:\widehat{C}\dashrightarrow \mathbb{P}^{n}$ given by\\

$[\alpha'\gamma(a_{1}+c_{1})z_{1}(x+\alpha y)(\gamma'x-\alpha'z)-\alpha\gamma'(b_{1}+c_{1})z_{2}(x+\alpha'y)(\gamma x-\alpha z)+d_{1}(\alpha-\alpha')(\gamma x-\alpha z)(\gamma'z-\alpha'x),
   \alpha'\gamma(a_{2}+c_{2})z_{1}(x+\alpha y)(\gamma'x-\alpha'z)-\alpha\gamma'(b_{2}+c_{2})z_{2}(x+\alpha'y)(\gamma x-\alpha z)+d_{2}(\alpha-\alpha')(\gamma x-\alpha z)(\gamma'z-\alpha'x),
   \cdots,
   \alpha'\gamma(a_{n-1}+c_{n-1})z_{1}(x+\alpha y)(\gamma'x-\alpha'z)-\alpha\gamma'(b_{n-1}+c_{n-1})z_{2}(x+\alpha'y)(\gamma x-\alpha z)+d_{n-1}(\alpha-\alpha')(\gamma x-\alpha z)(\gamma'z-\alpha'x),
   \alpha'\gamma z_{1}(x+\alpha y)(\gamma'x-\alpha'z)-\alpha\gamma'z_{2}(x+\alpha'y)(\gamma x-\alpha z)+d_{n}(\alpha-\alpha')(\gamma x-\alpha z)(\gamma'z-\alpha'x),
   (\alpha-\alpha')(\gamma x-\alpha x)(\gamma'x-\alpha'z)]$\\

extends to a morphism $f:C\longrightarrow\mathbb{P}^{n}$.
This gives a family of semi-stable maps of elliptic curves and after the stabilization we get a family of stable maps of elliptic curves over $S$.\\

Note that as in the case of degree 2, we can actually describe every element parametrized by $S$ and check it is an etale atlas. \\

(b)when essential part is smooth elliptic curve with j $\neq1728$.\\ Everything is same if we change the equation $\gamma-\alpha^{3}-\gamma^{2}\alpha-A\gamma^{3}$ above to $\gamma-\alpha^{3}-A\gamma^{2}\alpha-\gamma^{3}$.\\

(c)when essential part is singular curve.\\
Let $k$ be an algebraically closed field and  $a_{1}$, $a_{2}$, $\cdots$, $a_{n-1}$, $b_{1}$, $b_{2}$, $\cdots$, $b_{n-1}$, $c_{1}$, $c_{2}$, $\cdots$, $c_{n-1}$, $d_{1}$, $d_{2}$, $\cdots$, $d_{n}$, $z_{1}$, $z_{2}$, $A$, $\alpha$, $\beta$, $\alpha'$, $\beta'$ be indeterminates.
Let  $R$= k[$a_{1}$,$a_{2}$,$\cdots$,$a_{n-1}$,$b_{1}$,$b_{2}$,$\cdots$,$b_{n-1}$,$c_{1}$,$c_{2}$,$\cdots$,$c_{n-1}$,$d_{1}$,$d_{2}$,$\cdots$,$d_{n}$,$z_{1}$,$z_{2}$,$A$,$\alpha$,$\beta$,$\alpha'$,$\beta'$]/($a_{1}z_{1}-b_{1}z_{2}$,$a_{2}z_{1}-b_{2}z_{2}$,$\cdots$,$a_{n-1}z_{1}-b_{n-1}z_{2}$,$\beta^{2}-\alpha^{3}-\alpha-A$,$\beta'-\alpha'^{3}-\alpha'-A$)\

Let $D_{2}$, $D_{3}$, $F_{\alpha}$, $F_{\alpha'}$, $F_{\alpha-\alpha'}$, $G$ be ideals defined by
$(\alpha-\alpha',\beta-\beta')$,
$(z_{1},z_{2})$,\\
$(z_{1},b_{1},b_{2},\cdots,b_{n-1})$,
$(z_{2},a_{1},a_{2},\cdots,a_{n-1})$,
$(z_{1}-z_{2},a_{1}-b_{1},a_{2}-b_{2},\cdots,a_{n-1}-b_{n-1})$,
$(\frac{z_{1}}{z_{2}}(y+\beta)(x-\alpha')-(y+\beta')(x-\alpha))$.\\

Let $\widehat{S}=Spec(R)$, $S = Bl_{(\beta-\alpha-\frac{\alpha^{2}+\alpha\alpha'+\alpha'^{2}}{2},\beta'-\alpha'-\frac{\alpha^{2}+\alpha\alpha'+\alpha'^{2}}{2})}\widehat{S}$.\\
Let $\widehat{C}=Proj(R[x,y,z]/zy^{2}-x^{3}-x^{2}z-Az^{3})$, let $C'$ be a pull back of $\widehat{C}$, and $C$ be the blow up of $C'$ along ideals\\

$(y-x-\frac{x^{2}+\alpha x+\alpha^{2}}{2},\beta-\alpha-\frac{x^{2}+\alpha x+\alpha^{2}}{2})$,$(y-x-\frac{x^{2}+\alpha'x+\alpha'^{2}}{2},\beta'-\alpha'-\frac{x^{2}+\alpha'x+\alpha'^{2}}{2})$,

$(D_{2},x-\alpha z,y-\beta z)$,$(D_{2},x,z)$,

$(D_{3},x,z)$,$(D_{3},x-\alpha z,y-\beta z)$,$(D_{3},x-\alpha'z,y-\beta')$

$(F_{\alpha},x-\alpha z,y-\beta z)$,$(F_{\alpha'},x-\alpha'z,y-\beta'z)$,$(F_{\alpha-\alpha'},x,z)$.\\

Then rational map $\widehat{f}:\widehat{C}\dashrightarrow \mathbb{P}^{n}$ given by \\

$[(\alpha-\alpha')(a_{1}+c_{1})z_{1}(y+\beta z)(x-\alpha' z)-(\alpha-\alpha')(b_{1}+c_{1})z_{2}(y+\beta'z)(x-\alpha z)+d_{1}(\beta-\beta')(x-\alpha z)(x-\alpha'z),
   (\alpha-\alpha')(a_{2}+c_{2})z_{1}(y+\beta z)(x-\alpha' z)-(\alpha-\alpha')(b_{2}+c_{2})z_{2}(y+\beta'z)(x-\alpha z)+d_{1}(\beta-\beta')(x-\alpha z)(x-\alpha'z),
   \cdots,
   (\alpha-\alpha')(a_{n-1}+c_{n-1})z_{1}(y+\beta z)(x-\alpha' z)-(\alpha-\alpha')(b_{n-1}+c_{n-1})z_{2}(y+\beta'z)(x-\alpha z)+d_{n-1}(\beta-\beta')(x-\alpha z)(x-\alpha'z),
   (\alpha-\alpha')z_{1}(y+\beta z)(x-\alpha' z)-(\alpha-\alpha')z_{2}(y+\beta'z)(x-\alpha z)+d_{n}(\beta-\beta')(x-\alpha z)(x-\alpha'z),
   (\beta-\beta')(x-\alpha z)(x-\alpha'z)]$\\

 extends to a morphism $f:C\longrightarrow\mathbb{P}^{n}$.
This gives the family of semi-stable maps of elliptic curves and after the stabilization we get a family of stable map of elliptic curves over $S$.\\

\subsection{Local chart of $\overline{\mathbf{M}}_{1,0}^{log,ch}(\mathbb{P}^{n},3)$}
(a)when essential part is smooth elliptic curve with j $\neq0$.\\
In previous local chart $S$ of $\mathbf{\overline{M}}_{1,0}(\mathbb{P}^{2},3)_{0}$, the blow-up center of Vakil-Zinger desingularization is given by $D_{3}$. And $\sum_{1}$, $\sum_{2}$, $\Gamma_{1}$, $\Gamma_{2}$ are given by proper transforms of \\

$(D_{1},a_{1},a_{2},\cdots,a_{n-1},b_{1},b_{2},\cdots,b_{n-1}, \alpha z_{1}+\alpha'z_{2})$,

$(D_{2,\alpha},z_{1},b_{1},b_{2},\cdots,b_{n-1})(D_{2,\alpha'},z_{2},a_{1},a_{2},\cdots,a_{n-1})$,

$(D_{1},a_{1},a_{2},\cdots,a_{n-1},b_{1},b_{2},\cdots,b_{n-1})$,

$(D_{2,\alpha},a_{1},a_{2},\cdots,a_{n-1},b_{1},b_{2},\cdots,b_{n-1})(D_{2,\alpha'},a_{1},a_{2},\cdots,a_{n-1},b_{1},b_{2},\cdots,b_{n-1})$.\\

Let $\widetilde{S}$ be the blow up of S along $D_{3}$, $\sum_{2}$, $\Gamma_{2}$, $\sum_{1}$, $\Gamma_{1}$ and let $E_{1}$, $E_{2,\alpha}\bigcup E_{2,\alpha'}$, $L_{1}$, $L_{2,\alpha}\bigcup L_{2,\alpha'}$ be the exceptional divisors corresponding to $\sum_{1}$, $\sum_{2}$, $\Gamma_{1}$, $\Gamma_{2}$. Note that after blowing up along $\sum_{2}$, $\sum_{1}$ and $\Gamma_{2}$ are separated.
Now let $C^{''}$ be the pull back of $\widehat{C}$ along $\widetilde{S}$ and let $\widetilde{C}$ be the blow up of $C"$ along ideals\\

$(D_{1},x,z)$, $(L_{1},x,z)$,   $(E_{1},x,z)$,

$(D_{2,\alpha},x-\alpha'y,z-\gamma'y)$,$(L_{2,\alpha},x-\alpha'y,z-\gamma'y)$,$(E_{2,\alpha},x-\alpha'y,z-\gamma' y)$,$(D_{2,\alpha},x,z)$,\\
$(L_{2,\alpha},x,z)$,$(E_{2,\alpha},x,z)$,

$(D_{2,\alpha'},x-\alpha y,z-\gamma y)$,$(L_{2,\alpha'},x-\alpha y,z-\gamma y)$,$(E_{2,\alpha'},x-\alpha y,z-\gamma y)$,$(D_{2,\alpha'},x,z)$,\\
$(L_{2,\alpha'},x,z)$,$(E_{2,\alpha'},x,z)$,

$(D_{3},x,z)$,$(D_{3},x-\alpha y,z-\gamma y)$,$(D_{3},x-\alpha'y,z-\gamma'y)$,

$(F_{\alpha},x-\alpha y,z-\gamma y)$,$(F_{\alpha'},x-\alpha'y,z-\gamma'y)$,$(F_{\alpha-\alpha'},x,z)$,

$(\widetilde{L}_{1}^{2},G)$,
$(\widetilde{L}_{2,\alpha},G)$,$(\widetilde{L}_{2,\alpha'},G)$,\\
where $\widetilde{L}_{1}$,$\widetilde{L}_{2,\alpha}$,$\widetilde{L}_{2,\alpha}$ are exceptional divisor of $(L_{1},x,z)$,$(L_{2,\alpha},x,z)$,$(L_{2,\alpha'},x,z)$.\\

Let $\widetilde{W}$ be the blow-up of $\widetilde{S}\times \mathbb{P}^{2}$ along ideals\\

$(D_{3},x_{0}-d_{1}x_{n},x_{1}-d_{2}x_{n},\cdots,x_{n-1}-d_{n}x_{n}$),

$(E_{2}^{2},x_{0}-d_{1}x_{n},x_{1}-d_{2}x_{n},\cdots,x_{n-1}-d_{n}x_{n})$,
$(L_{2},x_{0}-d_{1}x_{n},x_{1}-d_{2}x_{n},\cdots,x_{n-1}-d_{n}x_{n})$,
$(D_{2},x_{0}-d_{1}x_{n},x_{1}-d_{2}x_{n},\cdots,x_{n-1}-d_{n}x_{n})$,

$(E_{1}^{3},x_{0}-d_{1}x_{n},x_{1}-d_{2}x_{n},\cdots,x_{n-1}-d_{n}x_{n})$,
$(L_{1}^{2},x_{0}-d_{1}x_{n},x_{1}-d_{2}x_{n},\cdots,x_{n-1}-d_{n}x_{n})$,
$(D_{1}^{2},x_{0}-d_{1}x_{n},x_{1}-d_{2}x_{n},\cdots,x_{n-1}-d_{n}x_{n})$,\\
where $x_{0}$,$x_{1}$,$\cdots$,$x_{n}$ are coordinates of $\mathbb{P}^{n}$.\\

Then $\widehat{f}:\widehat{C}\dashrightarrow \mathbb{P}^{n}$ extends to $\widetilde{f}:\widetilde{C}\longrightarrow\widetilde{W}$ and we get a family of admissible maps over $\widetilde{S}$.\\

(b)when the essential part is smooth elliptic curve with j $\neq1728$.\\ Everything is same if we change the equation $\gamma-\alpha^{3}-\gamma^{2}\alpha-A\gamma^{3}$ above to $\gamma-\alpha^{3}-A\gamma^{2}\alpha-\gamma^{3}$.\\

(c)when the essential part is singular curve.\\
In previous local chart $S$ of $\mathbf{\overline{M}}_{1,0}(\mathbb{P}^{2},3)_{0}$, the blow up center of Vakil-Zinger desingularization is given by $D_{3}$. And $\sum_{2}$, $\Gamma_{2}$ are given by proper transforms of $(D_{2},z_{1}-z_{2},a_{1}-b_{1})$, $(D_{2},a_{1},b_{1})$.

Let $\widetilde{S}$ be the blow up of $S$ along $D_{3}$, $\sum_{2}$, $\Gamma_{2}$ and let $E_{2}$,  $L_{2}$ be the exceptional divisors corresponding to $\sum_{2}$, $\Gamma_{2}$.
Now let $C''$ be the pull back of $C'$ along $\widetilde{S}$ and let $\widetilde{C}$ be the blow up of $C''$ along ideals\\

$(D_{2},x-\alpha z,y-\beta z)$,$(D_{2},x,z)$,$(L_{2},x-\alpha z,y-\beta z)$,$(L_{2},x,z)$,$(E_{2},x-\alpha z,y-\beta z)$,$(E_{2},x,z)$,

$(D_{3},x,z)$,$(D_{3},x-\alpha z,y-\beta z)$,$(D_{3},x-\alpha'z,y-\beta' z)$,

$(F_{\alpha},x-\alpha z,y-\beta z)$,$(F_{\alpha'},x-\alpha'z,y-\beta'z)$,$(F_{\alpha-\alpha'},x,z)$.

$(\widetilde{L}_{2},G)$,\\
where $\widetilde{L}_{2}$ is exceptional divisor of $(L_{2},x,z)$.\\

Let $\widetilde{W}$ be blow-up of $\widetilde{S}\times \mathbb{P}^{2}$ along ideal\\

$(D_{3},x_{0}-d_{1}x_{n},x_{1}-d_{2}x_{n},\cdots,x_{n-1}-d_{n}x_{n})$,

$(E_{2}^{2},x_{0}-d_{1}x_{n},x_{1}-d_{2}x_{n},\cdots,x_{n-1}-d_{n}x_{n})$,
$(L_{2},x_{0}-d_{1}x_{n},x_{1}-d_{2}x_{n},\cdots,x_{n-1}-d_{n}x_{n})$,
$(D_{2},x_{0}-d_{1}x_{n},x_{1}-d_{2}x_{n},\cdots,x_{n-1}-d_{n}x_{n})$.\\

Then $\widehat{f}:\widehat{C}\dashrightarrow \mathbb{P}^{n}$ extends to $\widetilde{f}:\widetilde{C}\longrightarrow\widetilde{W}$ and we get the family of admissible maps over $\widetilde{S}$.

\bigskip

\subsection{Main result}
 Let $\mathbf{\widehat{M}}$ be the blow up of $\mathbf{\widetilde{M}}_{1,0}(\mathbf{P}^{n},3)_{0}$ along $\sum_{2}$, $\Gamma_{2}$, $\sum_{1}$, $\Gamma_{1}$. By previous subsections, we can find a morphism from $\mathbf{\widehat{M}}$ to $\overline{\mathbf{M}}_{1,0}^{ch}(\mathbb{P}^{n},3)$. Here $\overline{\mathbf{M}}_{1,0}^{ch}(\mathbb{P}^{n},3)$ is moduli space of admissible stable maps of chain type without log structures. One can check that this morphism is finite surjective by using the result of section 5. Actually it is one to one morphism. On the other hand, We also have a finite surjective map from $\overline{\mathbf{M}}_{1,0}^{ch,log}(\mathbb{P}^{n},3)$ to $\overline{\mathbf{M}}_{1,0}^{ch}(\mathbb{P}^{n},3)$, which is just forgetting log structures(\cite{Kim}). By the uniqueness of the normalization, we get following theorem.\\

\noindent 
\textbf{Theorem 1.0.1.}
\textit{$\overline{\mathbf{M}}_{1,0}^{log,ch}(\mathbb{P}^{n},3)$ can be obtained by blowing-up $\mathbf{\widetilde{M}}_{1,0}(\mathbf{P}^{n},3)_{0}$ along the locus $\sum_{2}$, $\Gamma_{2}$, $\sum_{1}$, $\Gamma_{1}$.}

\bigskip

\remark Note that we only used the fact that forgetting morphism $\psi : \overline{\mathbf{M}}_{1,0}^{log,ch}(\mathbb{P}^{n},3)\longrightarrow\overline{\mathbf{M}}_{1,0}^{ch}(\mathbb{P}^{n},3)$ is a finite morphism. It follow from above that it is actually one to one morphism in our cases. We can also get this fact by calculating possible log structures. i.e. when $d\leqslant3$, there exists unique log structure on each admissible stable map. If $d\geqslant4$, there could be more than one log structures on one admissible stable map.\
\
\\
\\
\\
%%%%%%%%%%%%%%%%%%%%%%%%%%%%%%%%%%%%%%%%%%%%%%%%%%%%%%%%%%%%%%%%%%%%%%%%%%%%%%%%%%%%%%%%%%%%%%%%%%%%%%
%%%%%%%%%%%%%%%%%%%%%%%%%%%%%%%%%%%%%%%%%%%%%%%%%%%%%%%%%%%%%%%%%%%%%%%%%%%%%%%%%%%%%%%%%%%%%%%%%%%%%%
%%%%%%%%%%%%%%%%%%%%%%%%%%%%%%%%%%%%%%%%%%%%%%%%%%%%%%%%%%%%%%%%%%%%%%%%%%%%%%%%%%%%%%%%%%%%%%%%%%%%%%


\begin{thebibliography}{99}

\bibitem{Kim}B.\ Kim. \textit{Logarithmic stable maps}, arXiv:0807.3611v2.

\bibitem{KKO} B. Kim, A. Kresch, and Y.-G. Oh, {\em A compactification
of the space of maps from curves,} Preprint 2007.

\bibitem{Kato} K. Kato, {\em Logarithmic structures of Fontaine-Illusie,}
 Algebraic analysis, geometry, and number theory (Baltimore, MD, 1988), 191--224,
 Johns Hopkins Univ. Press, Baltimore, MD, 1989.

\bibitem{VZ} R. Vakil and A. Zinger,
{\em A Desingularization of the Main Component of the Moduli Space
of Genus-One Stable Maps into $\mathbb{P}^{n}$.} Geom. Topol.  12  (2008),  no. 1, 1--95.

\bibitem{VZ0} R. Vakil and A. Zinger, {\em A natural smooth
compactification of the space of elliptic curves in projective
space.} Electron. Res. Announc. Amer. Math. Soc.  13  (2007), 53--59.

\bibitem{HL} Y. Hu and J. Li, \emph{Genus-One Stable Maps, Local Equations and Vakil-ZingerÕs desingularization}, Math. Ann.  (2010). 

\bibitem{MOP} A.\ Marian, D.\ Oprea, and R.\ Pandharipande. \textit{The moduli space of stable quotients}, arXiv:0904.2992v2.

\bibitem{S} D.\ Smyth. \textit{Modular compactifications of $\mathbf{M}_{1,n}$ I}, arXiv:0808.0177v2.


\bibitem{Vis} M. Viscardi {\em Alternate compactifications of the moduli space of genus one maps}, arXiv : 1005.1431v1.

\bibitem{FP} W.\ Fulton and R.\ Pandharipande. \textit{Notes on stable maps and quantum cohomology}, Algebraic geometry - Santa Cruz 1995, 45-96, Proc.\ Sympos.\ Pure Math., 62, Part 2, Amer.\ Math.\ Soc., Providence, RI, 1997.

\end{thebibliography}
\end{document}